\documentclass[a4paper,pdftex,12pt]{article}

\usepackage{geometry}
\usepackage{times}
\usepackage{amssymb,amsmath}
\usepackage{epsfig, graphicx}

\newtheorem{Theorem}{Theorem}
\newtheorem{Corollary}[Theorem]{Corollary}
\newtheorem{Lemma}[Theorem]{Lemma}
\newtheorem{Definition}[Theorem]{Definition}
\newenvironment{proof}
{\begin{trivlist}\item[]{{\sc Proof.}}}{\hfill{$\square$}\noindent\end{trivlist}}

\title{A lower bound for $4$-regular planar unit distance graphs}
\author{Sascha Kurz$^\star$}
\date{\small $^\star$Fakult\"at f\"ur Mathematik, Physik und Informatik, Universit\"at Bayreuth, Germany, sascha.kurz@uni-bayreuth.de}

\newcommand{\lb}{$34$}

\begin{document}

\maketitle

{\small
\noindent \textbf{Abstract:}
We perform an exhaustive search for the minimum $4$-regular unit distance graph resulting in a lower bound of \lb~vertices. 
%The smallest known example, the so called Harborth graph, consists of $52$~vertices.
}

%\keywords{planar graphs, unit distances, quadrangulations, exhaustive enumeration, matchstick puzzle}
%\subjclass[2000]{05C10, 52A40}
%05C10 (Topological graph theory)
%52A40 (Geometric inequalities, etc. (convex geometry))
%\keywords[2000 \textit{Mathematics subject classification}]{05C10, 52A40}
%\keywords[\textit{Keywords and phrases}]{planar graphs, unit distances, quadrangulations, exhaustive enumeration, matchstick puzzle}

\section{Introduction}

\noindent
A graph is called a unit distance graph in the plane if there is an embedding of the vertex set into the plane such that every pair of adjacent vertices is at unit distance apart and the edges are non-crossing. Recognizing whether a given graph is a unit distance graph in the plane is NP-hard, see \cite{pre05493581}. In a companion paper \cite{companion} we haven given several necessary and computationally cheap criteria. % for a unit distance graph in the plane.
Here we consider the following geometric puzzle \cite{matchsticks_in_the_plane}: What is the minimum number of vertices of a planar unit distance graph where each vertex has degree $r$? 

Due to the Eulerian polyhedron formula for finite graphs only $r\le 5$ is possible. For $r=1,2$ the minimal examples are a single vertex and a triangle, respectively. The case $r=3$ is an entertaining amusement resulting in a minimum number of $8$~vertices. For $r=4$ the smallest known example is the so-called Harborth graph consisting of $52$~vertices, which has not been improved for more than twenty years. A vertex-degree of $r=5$ is not possible for a planar unit distance graph, as was recently proven in \cite{no_five_regular}.

We will verify that for $r=4$ at least \lb~vertices are necessary. To that end in \cite{companion} the computer program \texttt{plantri} was utilized to generate the $3$-connected planar $4$-regular graphs, where a graph $G$ is called $k$-connected if the graph remains connected when you delete fewer than $k$ vertices from the graph, and to show that up to $33$~vertices none of them is a unit distance graph in the plane. %So far the best known lower bound was  $20$ vertices \cite{no_five_regular}.

\section{$\mathbf{r}$-regular matchstick graphs}

\noindent
The planar $4$-regular unit distance graph with the minimum number of vertices is obviously $1$-connected, so that we consider possible subgraphs:

\begin{Definition}
  An \textbf{(incomplete) $\mathbf{r}$-regular matchstick graph} $\mathcal{M}$ consists of a connected planar graph $G=(V,E)$
  and a straight-line embedding %$f:V\rightarrow\mathbb{R}^2$
  in the plane such that all edges have length $1$ and are non-intersecting. Additionally the nodes on the outer face of
  $\mathcal{M}$ all have degree at most $r$ and all other nodes have degree exactly $r$.
\end{Definition}

\begin{Definition}
  We abbreviate the number $|V|$ of vertices by $n(\mathcal{M})$ and by $\mathcal{K}(\mathcal{M})$ we denote the set of vertices which is situated on
  the outer face of $\mathcal{M}$.  The set of the remaining vertices is denoted by $\mathcal{I}(\mathcal{M})$. For the
  cardinality of $\mathcal{K}(\mathcal{M})$ we write $k(\mathcal{M})$. By
  $\tau(\mathcal{M})$ we denote the quantity
  $r\cdot k(\mathcal{M})-\sum\limits_{v\in\mathcal{K}(\mathcal{M})}\delta(v)$, where $\delta(v)$ denotes the
  degree of vertex $v$. By $A_i(\mathcal{M})$ we denote the number of faces of $\mathcal{M}$ which are
  $i$-gons. Here we also count the outer face.
\end{Definition}

Whenever it is clear from the context which matchstick graph $\mathcal{M}$ is meant we only write $n$, $\mathcal{K}$, $\mathcal{I}$, $k$, $\tau$, $A_i$. 
Using $|V|-|E|+|F|=2$ we get:

\begin{Lemma}
  For an (incomplete) $4$-regular matchstick graph $\mathcal{M}$ we have
  \label{lemma_A_i_sum}
  \begin{eqnarray}
    |E|&=&\frac{1}{2}\cdot\sum_{i=3}^\infty i\cdot A_i,\quad\quad|V|=\frac{1}{4}\cdot\sum_{i=3}^\infty i\cdot A_i\,+\frac{\tau}{4}, \nonumber\\
    |F|&=&\sum\limits_{i=3}^\infty A_i,\,\,\,\quad\quad\quad\quad n=|V|=|F|-2+\frac{\tau}{2},\nonumber\\
    8-\tau&=&\sum_{i=3}^\infty (4-i) A_i=A_3-A_5-2A_6-3A_7-\dots,\label{eq_A_i_sum}\\
    A_3&\ge& 4+k-\tau. \label{eq_force_triangles}
  \end{eqnarray}
\end{Lemma}

We remark that the parameter $\tau$ must be an even integer. Thus a $4$-regular matchstick graph cannot contain a bridge. In the following lemma we will show that a complete, i.~e.{} $\tau=0$, $4$-regular matchstick graph which is not $3$-connected must contain some $4$-regular matchstick graphs with $\tau\le 4$ as subgraphs. Let $\mathcal{M}_\tau$ denote a $4$-regular matchstick graph with given $\tau$, minimum degree $2$, and the minimum number of vertices. For $\tau\ge 6$ we have $n\left(\mathcal{M}_\tau\right)=\frac{\tau}{2}$, where the unique example is given by $C_{\frac{\tau}{2}}$ -- a simple cycle consisting of $\frac{\tau}{2}$ vertices. Since deleting an edge of the outer face decreases $\tau$ by two, we have
\begin{equation}
  \label{eq_monotony}
  n\left(\mathcal{M}_0\right)\ge n\left(\mathcal{M}_2\right)\ge n\left(\mathcal{M}_4\right)\ge 3.
\end{equation}
\begin{Lemma}
  If $\mathcal{M}$ is a $4$-regular matchstick graph with $\tau=2$, then $n(\mathcal{M})
  \ge n\left(\mathcal{M}_2\right)$. If $\mathcal{M}'$ is a $4$-regular matchstick graph with $\tau=4$,
  then $\mathcal{M}'$ consists either of a single vertex or we have $n(\mathcal{M}')
  \ge n\left(\mathcal{M}_4\right)$. 
\end{Lemma}
\begin{proof}
  Deleting a vertex of degree $1$ in a $4$-regular matchstick graph decreases $\tau$ by two. Thus $\mathcal{M}$ can not
  contain a vertex of degree $1$ and $\mathcal{M}'$ can contain at most one vertex of degree $1$.
\end{proof}

%We can use lower bounds on $n\left(\mathcal{M}_2\right)$ and $n\left(\mathcal{M}_4\right)$
%to deduce a lower bound on the number of vertices of a complete $4$-regular matchstick graph which is not $3$-connected:

\begin{Lemma}
  \label{lemma_not_3_connected}
  If $\mathcal{M}$ is a complete $4$-regular matchstick graph which is not $3$-connected, then we have
  $$
    n(\mathcal{M})\ge\min\Big(2\cdot n\left(\mathcal{M}_2\right),
    2\cdot n\left(\mathcal{M}_4\right)+2\Big).
  $$
\end{Lemma}
\begin{proof}
  If deleting a vertex $x$ from $\mathcal{M}$ results in $s$ matchstick graphs $\mathcal{C}_1,\dots,\mathcal{C}_s$,
  then we have
  $
    0=\tau(\mathcal{M})=\sum\limits_{i=1}^s\tau\left(\mathcal{C}_i\right)\,-\,4
  $.
  In this case only $s=2$, $\tau\left(\mathcal{C}_i\right)=2$ is possible, which results in
  $n(\mathcal{M})\ge 2\cdot n\left(\mathcal{M}_2\right)+1$.
  
  In the remaining part of the proof we assume that $\mathcal{M}$ is $2$-connected. Let $\{x,y\}$ be an edge in $\mathcal{M}$.
  If deleting vertices $x$ and $v$ from $\mathcal{M}$ results in $s$ matchstick graphs $\mathcal{C}_1,\dots,\mathcal{C}_s$,
  then we have
  $
    0=\tau(\mathcal{M})=\sum\limits_{i=1}^s\tau\left(\mathcal{C}_i\right)\,-\,6
  $.
  
  Thus only $s=3$, $\tau\left(\mathcal{C}_i\right)=2$ or
  $s=2$, $\left\{\tau\left(\mathcal{C}_1\right),\tau\left(\mathcal{C}_2\right)\right\}=\{2,4\}$, where no $\mathcal{C}_i$ consists of a single vertex, 
  are possible. Due to Equation~(\ref{eq_monotony}) we have
  $n(\mathcal{M})\ge 2\cdot n\left(\mathcal{M}_4\right)+2$ in this case.
  
  It remains to consider the deletion of vertices $x$, $y$ where $\{x,y\}$ is not an edge of $\mathcal{M}$. Here we have
  $
    0=\tau(\mathcal{M})=\sum\limits_{i=1}^s\tau\left(\mathcal{C}_i\right)\,-\,8
  $,
  using the notation from above. If $s\ge 2$ then at least two $\mathcal{C}_i$ have a $\tau$ equal to $2$, so that
  the proposed inequality is true. For $s=2$ only $\tau\left(\mathcal{C}_i\right)=4$,
  where no $\mathcal{C}_i$ consists of a single vertex, or $\left\{\tau\left(\mathcal{C}_1\right),\tau\left(\mathcal{C}_2\right)\right\}=\{2,6\}$ are
  possible. In the latter case we consider the neighbors $x'$ and $y'$ of $x$ and $y$ in the connectivity component 
  with $\tau=2$. Deleting the edges $\{x,x'\}$ and $\{y,y'\}$ from $\mathcal{M}$ results in two graphs $\mathcal{C}_1$, $\mathcal{C}_2$ with $\tau\left(\mathcal{C}_i\right)=2$.
\end{proof}

Analogously we can prove:
\begin{Lemma}
  \label{lemma_m_tau_two_connected}
  The $4$-regular matchstick graphs $\mathcal{M}_2$ and $\mathcal{M}_4$ are $2$-connected.
\end{Lemma}

For the determination of lower bounds for $n\left(\mathcal{M}_2\right)$ and $n\left(\mathcal{M}_4\right)$ we have a special but effective exclusion criterion:
\begin{Lemma}
  \label{lemma_smallest_example}
  Let $\mathcal{M}$ be a $4$-regular matchstick graph with $\tau(\mathcal{M})\in\{2,4\}$. If $\mathcal{M}$ contains a vertex of
  degree $2$ on the outer face which is adjacent to an inner triangle, then
  $n(\mathcal{M})>n\left(\mathcal{M}_{\tau(\mathcal{M})}\right)$.
\end{Lemma}
\begin{proof}
  Consider the matchstick graph arising from $\mathcal{M}$ after deleting vertex $v$.
\end{proof}

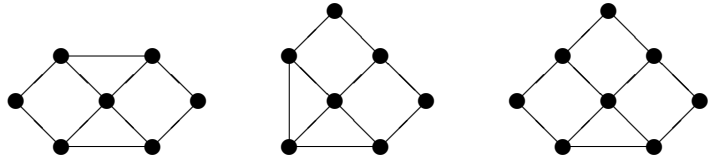
\begin{figure}[h]
  \begin{center}
    \setlength{\unitlength}{0.60cm}
    \begin{picture}(15,3)
      \put(1,0){\line(1,0){2}}
      \put(1,0){\line(1,1){1}}
      \put(3,0){\line(-1,1){1}}
      \put(2,1){\line(-1,1){1}}
      \put(2,1){\line(1,1){1}}
      \put(1,2){\line(1,0){2}}
      \put(1,0){\line(-1,1){1}}
      \put(0,1){\line(1,1){1}}
      \put(4,1){\line(-1,1){1}}
      \put(3,0){\line(1,1){1}}
      \put(1,0){\circle*{0.35}}
      \put(3,0){\circle*{0.35}}
      \put(1,2){\circle*{0.35}}
      \put(3,2){\circle*{0.35}}
      \put(2,1){\circle*{0.35}}
      \put(0,1){\circle*{0.35}}
      \put(4,1){\circle*{0.35}}
      \put(6,0){\line(0,1){2}}
      \put(6,0){\line(1,0){2}}
      \put(8,0){\line(-1,1){2}}
      \put(6,0){\line(1,1){1}}
      \put(8,0){\line(1,1){1}}
      \put(7,1){\line(1,1){1}}
      \put(6,2){\line(1,1){1}}
      \put(7,3){\line(1,-1){2}}
      \put(6,0){\circle*{0.35}}
      \put(6,2){\circle*{0.35}}
      \put(8,0){\circle*{0.35}}
      \put(7,1){\circle*{0.35}}
      \put(7,3){\circle*{0.35}}
      \put(8,2){\circle*{0.35}}
      \put(9,1){\circle*{0.35}}
      \put(12,0){\line(1,0){2}}
      \put(12,0){\line(-1,1){1}}
      \put(11,1){\line(1,1){1}}
      \put(14,0){\line(-1,1){2}}
      \put(12,0){\line(1,1){1}}
      \put(14,0){\line(1,1){1}}
      \put(13,1){\line(1,1){1}}
      \put(12,2){\line(1,1){1}}
      \put(13,3){\line(1,-1){2}}
      \put(11,1){\circle*{0.35}}
      \put(12,0){\circle*{0.35}}
      \put(12,2){\circle*{0.35}}
      \put(14,0){\circle*{0.35}}
      \put(13,1){\circle*{0.35}}
      \put(13,3){\circle*{0.35}}
      \put(14,2){\circle*{0.35}}
      \put(15,1){\circle*{0.35}}
    \end{picture}\\[2mm]
    \caption{Configurations of triangles and quadrangles.}
    \label{fig_configuration_1}
  \end{center}
\end{figure}

Some necessary criteria for matchstick graphs are stated in \cite{companion}. 
\begin{Lemma}\textbf{(\textit{Area argument})}
  \label{lemma_area_argument}
  $$
    \frac{\frac{k}{4}\cdot\cot\left(\frac{\pi}{k}\right)}{\frac{\sqrt{3}}{4}}\ge \sum\limits_{i=1}^\infty A_{2i+1}\,+\,2\cdot\text{ \# disjoint configurations as in
    Figure~\ref{fig_configuration_1}}
  $$
\end{Lemma}

Combining Lemma~\ref{lemma_area_argument} and Lemma~\ref{lemma_A_i_sum} yields:
\begin{Corollary}
  \label{cor_k_bound_inner_vertices}
  For a $4$-regular matchstick graph $\mathcal{M}$ with $|\mathcal{I}|>0$ we have
  $$
    n\ge 5+k-\frac{\tau}{2}\quad\text{and}\quad
    \frac{A_\text{max}(k)}{A_\text{max}(3)}=\frac{\frac{k}{4}\cdot\cot\left(\frac{\pi}{k}\right)}{\frac{\sqrt{3}}{4}}\ge\frac{Area(\mathcal{M})}{\frac{\sqrt{3}}{4}} \ge 6+k-\tau.
  $$
\end{Corollary}

\section{Lower bounds on $\mathbf{n\left(\mathcal{M}_\tau\right)}$ for $\mathbf{\tau\le 4}$}

\noindent
%Using the ideas from the previous section we can prove some bounds on the parameters of $4$-regular matchstick graphs with given $\tau$. 
If $\tau<6$ then $\mathcal{M}$ is either a single vertex or contains inner vertices.

\begin{Lemma}
  \label{lemma_first_tau_bounds}
  Let $\mathcal{M}$ be a $4$-regular matchstick graph, then we have
  the following bounds on its parameters
  \begin{eqnarray*}
     k\ge 11,\, A_3\ge 15,\, \frac{Area(\mathcal{M})}{\frac{\sqrt{3}}{4}}\ge 17,\,n\ge 16
     &\text{for }\tau=0,\\
     k\ge 9,\, A_3\ge 11,\, \frac{Area(\mathcal{M})}{\frac{\sqrt{3}}{4}}\ge 13,\, n\ge 13
     &\text{for }\tau=2,\\
     n= 1\text{ or }k\ge 8,\, A_3\ge 8,
     \frac{Area(\mathcal{M})}{\frac{\sqrt{3}}{4}}\ge 10, n\ge 11 &\text{for }\tau=4.\\
  \end{eqnarray*}
\end{Lemma}
\begin{proof}
  The case that $\mathcal{M}$ equals a single vertex is only possible for $\tau=4$. Otherwise we have $k\ge 3$ and $\mathcal{M}$ contains at least one inner vertex. From 
  Corollary~\ref{cor_k_bound_inner_vertices} and Lemma~\ref{lemma_area_argument} we deduce the bounds $k\ge 10$ for $\tau=0$,
  $k\ge 9$ for $\tau=2$, and $k\ge 8$ for $\tau=4$. In an extra consideration we conclude a contradiction from $\tau=0$,
  $k=10$ so that we have $k\ge 11$ in this case. Inserting this into Inequality~(\ref{eq_force_triangles}) and
  Corollary~\ref{cor_k_bound_inner_vertices} yields the remaining bounds.
  
  Let us assume $\tau(\mathcal{M})=0$ and $k(\mathcal{M})=10$. Due to Corollary~\ref{cor_k_bound_inner_vertices},
  Equation~(\ref{eq_A_i_sum}), and the area argument we have the following possible non-zero values for the $A_i$:
  \begin{enumerate}
    \item[(1)] $A_3=14$, $A_4=t\in\mathbb{N}$, $A_{10}=1$,
    \item[(2)] $A_3=15$, $A_4=t\in\mathbb{N}$, $A_5=1$, $A_{10}=1$, or
    \item[(3)] $A_3=16$, $A_4=t\in\mathbb{N}$, $A_6=1$, $A_{10}=1$.
  \end{enumerate}
  Due to an angle sum of $(10-2)\pi$ at most $3\cdot 10-7=23$ of the inner angles of the outer face can be part
  of triangles. In a similar manner we conclude that at most $3s-3$ outer angles of an inner $s$-gon can be part
  of triangles. Since $15\cdot 3-23 -(3\cdot 5-3)=10>0$ and $16\cdot 3-23 -(3\cdot 6-3)=10>0$ one of the
  configurations of Figure~\ref{fig_configuration_1} must exist in cases (2) and (3), which is a contradiction to the area argument.
  
  The number of outer and inner angles of one of the configurations in Figure~\ref{fig_configuration_1}, which can be part
  of triangles, is at most $15$. Thus we can conclude from $14\cdot 3-23-1\cdot 15=4>0$  that in case (1) 
  at least two subgraphs as in Figure~\ref{fig_configuration_1} must exist, which contradicts the area argument.
\end{proof}

\begin{figure}[htp]
  \begin{center}
    \setlength{\unitlength}{0.6cm}
    \begin{picture}(4,2.5)
      \qbezier[20](1,0)(2,0)(3,0)
      \put(1.2,-0.4){$x$}
      \put(2.5,-0.4){$y$}
      \put(0,0.5){\line(0,1){1}}
      \put(1,0){\line(0,1){2}}
      \put(3,0){\line(0,1){2}}
      \put(4,0.5){\line(0,1){1}}
      \put(0,0.5){\line(2,1){1}}
      \put(0,1.5){\line(2,1){1}}
      \put(0,1.5){\line(2,-1){1}}
      \put(0,0.5){\line(2,-1){1}}
      \put(1,2){\line(1,0){2}}
      \put(1,2){\line(1,-2){0.5}}
      \put(2,2){\line(1,-2){0.5}}
      \put(2,2){\line(-1,-2){0.5}}
      \put(3,2){\line(-1,-2){0.5}}
      \put(1.5,1){\line(1,0){1}}
      \put(3,2){\line(2,-1){1}}
      \put(3,1){\line(2,-1){1}}
      \put(3,1){\line(2,1){1}}
      \put(3,0){\line(2,1){1}}
      \put(1.5,1){\line(1,-1){0.5}}
      \put(2.5,1){\line(-1,-1){0.5}}
      \put(2,0.5){\line(-2,-1){1}}
      \put(2,0.5){\line(2,-1){1}}
      \put(2,0.5){\circle*{0.35}}
      \put(1,0){\circle*{0.35}}
      \put(1,1){\circle*{0.35}}
      \put(1,2){\circle*{0.35}}
      \put(0,0.5){\circle*{0.35}}
      \put(0,1.5){\circle*{0.35}}
      \put(1.5,1){\circle*{0.35}}
      \put(2,2){\circle*{0.35}}
      \put(2.5,1){\circle*{0.35}}
      \put(3,2){\circle*{0.35}}
      \put(3,1){\circle*{0.35}}
      \put(3,0){\circle*{0.35}}
      \put(4,0.5){\circle*{0.35}}
      \put(4,1.5){\circle*{0.35}}
    \end{picture}
    \quad\quad\quad\quad
    \begin{picture}(3,2.5)
      \put(1.2,-0.4){$x$}
      \put(1.2,2.2){$y$}
      \put(0,0.5){\line(0,1){1}}
      \put(1,0){\line(0,1){2}}
      \put(0,0.5){\line(2,1){1}}
      \put(0,1.5){\line(2,1){1}}
      \put(0,1.5){\line(2,-1){1}}
      \put(0,0.5){\line(2,-1){1}}
      \put(1,2){\line(1,0){2}}
      \put(1,0){\line(1,0){2}}
      \put(1,2){\line(1,-1){0.5}}
      \put(2,2){\line(1,-1){0.5}}
      \put(2,2){\line(-1,-1){0.5}}
      \put(3,2){\line(-1,-1){0.5}}
      \put(1,0){\line(1,1){0.5}}
      \put(2,0){\line(1,1){0.5}}
      \put(2,0){\line(-1,1){0.5}}
      \put(3,0){\line(-1,1){0.5}}
      \put(1.5,1.5){\line(1,0){1}}
      \put(1.5,0.5){\line(1,0){1}}
      \put(1.5,0.5){\line(0,1){1}}
      \put(3,0){\line(0,1){2}}
      \put(3,1){\line(-1,-1){0.5}}
      \put(3,1){\line(-1,1){0.5}} 
      \put(3,1){\circle*{0.35}}
      \put(1.5,1.5){\circle*{0.35}}
      \put(2.5,1.5){\circle*{0.35}}
      \put(1.5,0.5){\circle*{0.35}}
      \put(2.5,0.5){\circle*{0.35}}
      \put(1,0){\circle*{0.35}}
      \put(1,1){\circle{0.35}}
      \put(2,2){\circle*{0.35}}
      \put(3,2){\circle*{0.35}}
      \put(2,0){\circle*{0.35}}
      \put(3,0){\circle*{0.35}}
      \put(1,2){\circle*{0.35}}
      \put(0,0.5){\circle{0.35}}
      \put(0,1.5){\circle{0.35}}
    \end{picture}\\[1mm]
    \caption{Examples of non unit distance graphs.}
    \label{fig_example_4}
  \end{center}
\end{figure}

To obtain sharper lower bounds for $n\left(\mathcal{M}_2\right)$ and $n\left(\mathcal{M}_4\right)$ we designed a recursive algorithm to exhaustively generate $4$-regular matchstick graphs. Starting from a triangle in each iteration a face is added. Those planar graphs which do not pass the necessary criteria from \cite{companion} or Lemma~\ref{lemma_smallest_example} are removed. For $n\le 16$ there only remain the five planar graphs from Figures \ref{fig_example_4}, \ref{fig_example_5}, and \ref{fig_example_6}.

\begin{figure}[htp]
  \begin{center}
    \setlength{\unitlength}{0.6cm}
    \begin{picture}(5.5,2.5)
      \put(0,1){\line(1,0){2}}
      \put(3,1){\line(1,0){2}}
      \put(0.5,0){\line(1,0){4}}
      \put(0.5,0){\line(-1,2){0.5}}
      \put(1.5,0){\line(-1,2){0.5}}
      \put(2.5,0){\line(-1,2){0.5}}
      \put(3.5,0){\line(-1,2){0.5}}
      \put(4.5,0){\line(-1,2){0.5}}
      \put(0.5,0){\line(1,2){0.5}}
      \put(1.5,0){\line(1,2){0.5}}
      \put(2.5,0){\line(1,2){0.5}}
      \put(3.5,0){\line(1,2){0.5}}
      \put(4.5,0){\line(1,2){0.5}}
      \put(2,2){\line(0,-1){1}}
      \put(2,2){\line(1,-1){1}}
      \put(2,2){\line(-2,-1){2}}
      \put(1,1.5){\line(-1,0){1}}
      \put(0,2){\line(1,0){2}}
      \qbezier[50](0,2)(5,3.5)(5,1)
      \put(-0.55,1.9){$x$}
      \put(5.2,0.9){$y$}
      \put(1,1.5){\line(0,1){0.5}}
      \put(0,1.5){\line(0,-1){0.5}}
      \put(0,1.5){\line(0,1){0.5}}
      \put(0,1.5){\line(2,1){1}}
      \put(1,1.5){\circle*{0.35}}
      \put(0,1.5){\circle*{0.35}}
      \put(1,2){\circle*{0.35}}
      \put(0,2){\circle*{0.35}}
      \put(0,1){\circle*{0.35}}
      \put(1,1){\circle*{0.35}}
      \put(2,1){\circle*{0.35}}
      \put(3,1){\circle*{0.35}}
      \put(4,1){\circle*{0.35}}
      \put(5,1){\circle*{0.35}}
      \put(0.5,0){\circle*{0.35}}
      \put(1.5,0){\circle*{0.35}}
      \put(2.5,0){\circle*{0.35}}
      \put(3.5,0){\circle*{0.35}}
      \put(4.5,0){\circle*{0.35}}
      \put(2,2){\circle*{0.35}}
    \end{picture}
    \quad\quad\quad\quad
    \begin{picture}(5.5,2.5)
      \put(0,1){\line(1,0){2}}
      \put(3,1){\line(1,0){2}}
      \put(0.5,0){\line(1,0){4}}
      \put(0.5,0){\line(-1,2){0.5}}
      \put(1.5,0){\line(-1,2){0.5}}
      \put(2.5,0){\line(-1,2){0.5}}
      \put(3.5,0){\line(-1,2){0.5}}
      \put(4.5,0){\line(-1,2){0.5}}
      \put(0.5,0){\line(1,2){0.5}}
      \put(1.5,0){\line(1,2){0.5}}
      \put(2.5,0){\line(1,2){0.5}}
      \put(3.5,0){\line(1,2){0.5}}
      \put(4.5,0){\line(1,2){0.5}}
      \put(2,2){\line(0,-1){1}}
      \put(2,2){\line(1,-1){1}}
      \put(2,2){\line(-2,-1){2}}
      \put(1,1.5){\line(-1,0){1}}
      \put(1,2){\line(1,0){1}}
      \qbezier[50](1,2)(5,3.5)(5,1)
      \put(0.45,1.95){$x$}
      \put(5.2,0.9){$y$}
      \put(1,1.5){\line(0,1){0.5}}
      \put(0,1.5){\line(0,-1){0.5}}
      \put(0,1.5){\line(2,1){1}}
      \put(1,1.5){\circle*{0.35}}
      \put(0,1.5){\circle*{0.35}}
      \put(1,2){\circle*{0.35}}
      \put(0,1){\circle*{0.35}}
      \put(1,1){\circle*{0.35}}
      \put(2,1){\circle*{0.35}}
      \put(3,1){\circle*{0.35}}
      \put(4,1){\circle*{0.35}}
      \put(5,1){\circle*{0.35}}
      \put(0.5,0){\circle*{0.35}}
      \put(1.5,0){\circle*{0.35}}
      \put(2.5,0){\circle*{0.35}}
      \put(3.5,0){\circle*{0.35}}
      \put(4.5,0){\circle*{0.35}}
      \put(2,2){\circle*{0.35}}
    \end{picture}
    \caption{Examples of non unit distance graphs.}
    \label{fig_example_5}
  \end{center}
\end{figure}

\begin{Lemma}
  None of the planar graphs from Figures \ref{fig_example_4}, \ref{fig_example_5}, and \ref{fig_example_6} is a matchstick graph.
\end{Lemma}
\begin{proof}
  If we determine the (up to symmetries) unique coordinates of the planar graph without the dotted edge 
  on the left hand side in Figure~\ref{fig_example_4}, then the distance between vertices $x$ and $y$ is not equal to $1$.
  In the graph given by the filled circles on the right hand side of Figure~\ref{fig_example_4} we can determine unique
  coordinates and conclude that the distance between the vertices $x$ and $y$ is not equal to $2$.
  
  If we determine the unique coordinates of the two planar graphs without the dotted edge in Figure~\ref{fig_example_5},
  then in both case the distance between vertex $x$ and $y$ is not equal to $1$.
  
  By considering the angles in the planar graph of Figure~\ref{fig_example_5} we deduce $\alpha+\beta=\frac{\pi}{3}$.
  With these two angles the squared distance between vertices $x$ and $y$ is given by 
  $4-4\sqrt{3}\cdot\cos\left(\frac{\alpha-\beta}{2}\right)$.
  The minimum distance, which equals $1$, is attained for $\alpha=0$ and $\beta=0$.
\end{proof}

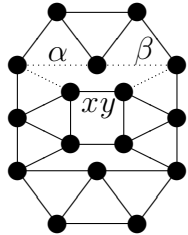
\begin{figure}[htp]
  \begin{center}
    \setlength{\unitlength}{0.7cm}
    \begin{picture}(5.5,4.5)
      \put(0.57,3.07){$\alpha$}
      \put(2.20,3.12){$\beta$}
      \put(1.2,2.12){$x$}
      \put(1.55,2.12){$y$}
      \put(0,3){\line(3,4){0.75}}
      \put(1.5,3){\line(3,4){0.75}}
      \put(1.5,3){\line(-3,4){0.75}}
      \put(3,3){\line(-3,4){0.75}}
      \put(0,1){\line(3,-4){0.75}}
      \put(1.5,1){\line(3,-4){0.75}}
      \put(1.5,1){\line(-3,-4){0.75}}
      \put(3,1){\line(-3,-4){0.75}}
      \put(0,1){\line(2,1){1}}
      \put(0,2){\line(2,1){1}}
      \put(0,2){\line(2,-1){1}}
      \qbezier[15](0,3)(0.5,2.75)(1,2.5)
      %\put(0,3){\line(2,-1){1}}
      \put(3,1){\line(-2,1){1}}
      \put(3,2){\line(-2,1){1}}
      \put(3,2){\line(-2,-1){1}}
      \qbezier[15](3,3)(2.5,2.75)(2,2.5)
      %\put(3,3){\line(-2,-1){1}}
      \put(0.75,0){\line(1,0){1.5}}
      \put(0.75,4){\line(1,0){1.5}}
      \put(0,1){\line(1,0){3}}
      \put(0,1){\line(0,1){2}}
      \put(3,1){\line(0,1){2}}
      \put(1,1.5){\line(1,0){1}}
      \put(1,2.5){\line(1,0){1}}
      \put(1,1.5){\line(0,1){1}}
      \put(2,1.5){\line(0,1){1}}
      \qbezier[30](0,3)(1.5,3)(3,3)
      %\put(0,3){\line(1,0){3}}
      %
      \put(0.75,0){\circle*{0.35}}
      \put(2.25,0){\circle*{0.35}}
      \put(0.75,4){\circle*{0.35}}
      \put(2.25,4){\circle*{0.35}}
      \put(0,1){\circle*{0.35}}
      \put(1.5,1){\circle*{0.35}}
      \put(3,1){\circle*{0.35}}
      \put(1,1.5){\circle*{0.35}}
      \put(2,1.5){\circle*{0.35}}
      \put(1,2.5){\circle*{0.35}}
      \put(2,2.5){\circle*{0.35}}
      \put(0,2){\circle*{0.35}}   
      \put(0,3){\circle*{0.35}}
      \put(3,2){\circle*{0.35}}   
      \put(3,3){\circle*{0.35}}
      \put(1.5,3){\circle*{0.35}}
    \end{picture}
    \caption{Example of a non unit distance graph.}
    \label{fig_example_6}
  \end{center}
\end{figure}

\begin{Corollary}
  \label{col_M_tau}
  $
    n\left(\mathcal{M}_2\right)\ge n\left(\mathcal{M}_4\right)\ge 17
  $.
\end{Corollary}

With the unit distance graph in Figure~\ref{fig_example_7}  we have $17\le n\left(\mathcal{M}_4\right)\le 20$. Maybe the determination of $n\left(\mathcal{M}_4\right)$ and $n\left(\mathcal{M}_2\right)$ is an interesting  problem of its own.

\begin{figure}[htp]
  \begin{center}
    \setlength{\unitlength}{0.6cm}
    \begin{picture}(5.5,4.5)
      \put(2.5,0){\line(-1,1){0.5}}
      \put(2,0.5){\line(-2,1){2}}
      \put(2.5,0){\line(1,1){0.5}}
      \put(3,0.5){\line(2,1){2}}
      \put(2.5,3){\line(-1,-1){0.5}}
      \put(2,2.5){\line(-2,-1){2}}
      \put(2.5,3){\line(1,-1){0.5}}
      \put(3,2.5){\line(2,-1){2}}
      \put(2,0.5){\line(0,1){2}}
      \put(3,0.5){\line(0,1){2}}
      \put(0,0.5){\line(0,1){2}}
      \put(5,0.5){\line(0,1){2}}
      \put(1,0){\line(1,0){3}}
      \put(1,3){\line(1,0){3}}
      \put(0,0.5){\line(2,1){1}}
      \put(0,0.5){\line(2,-1){1}}
      \put(0,2.5){\line(2,1){1}}
      \put(0,2.5){\line(2,-1){1}}
      \put(5,0.5){\line(-2,1){1}}
      \put(5,0.5){\line(-2,-1){1}}
      \put(5,2.5){\line(-2,1){1}}
      \put(5,2.5){\line(-2,-1){1}}
      \put(1,0){\line(0,1){1}}
      \put(1,2){\line(0,1){1}}
      \put(4,0){\line(0,1){1}}
      \put(4,2){\line(0,1){1}}
      \put(1,0){\line(2,1){1}}
      \put(1,3){\line(2,-1){1}}
      \put(4,0){\line(-2,1){1}}
      \put(4,3){\line(-2,-1){1}}
      \put(0,0.5){\circle*{0.35}}
      \put(0,1.5){\circle*{0.35}}
      \put(1,0){\circle*{0.35}}
      \put(1,1){\circle*{0.35}}
      \put(2,0.5){\circle*{0.35}}
      \put(2.5,0){\circle*{0.35}}
      \put(5,0.5){\circle*{0.35}}
      \put(5,1.5){\circle*{0.35}}
      \put(4,0){\circle*{0.35}}
      \put(4,1){\circle*{0.35}}
      \put(3,0.5){\circle*{0.35}}
      \put(0,2.5){\circle*{0.35}}
      \put(1,3){\circle*{0.35}}
      \put(1,2){\circle*{0.35}}
      \put(2,2.5){\circle*{0.35}}
      \put(2.5,3){\circle*{0.35}}
      \put(5,2.5){\circle*{0.35}}
      \put(4,3){\circle*{0.35}}
      \put(4,2){\circle*{0.35}}
      \put(3,2.5){\circle*{0.35}}
    \end{picture}
    \caption{Example for $\tau=4$ consisting of $20$ vertices.}
    \label{fig_example_7}
  \end{center}
\end{figure}
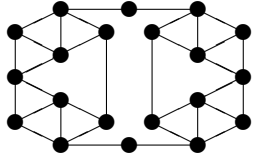

\bigskip

If $n\left(\mathcal{M}_0\right)<34$ then $\mathcal{M}_0$ has to be $3$-connected due to Lemma~\ref{lemma_not_3_connected}. This possibility is excluded by \cite{companion} so that we can conclude:

\begin{Theorem}
  \label{main_theorem}
  If $\mathcal{M}$ is a $4$-regular matchstick graph with $\tau=0$, then it has at least \lb~vertices.
\end{Theorem}

\section{Complete classification of $\mathbf{4}$-regular matchstick graphs for small $\mathbf{k}$}

\noindent
Another source of candidates for planar unit distance graphs arises if one tries to classify all $4$-regular matchstick graphs $\mathcal{M}$ with small cardinality $k(\mathcal{M})$ of its outer face. Using the described methods it is not too hard to classify these graphs for $k\le 6$.

\begin{Lemma}
  The complete list of $4$-regular matchstick graphs with $k\in\{3,4\}$ is given by\\[-1.2cm]
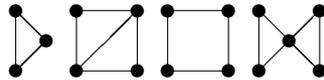
  \begin{center}
    \setlength{\unitlength}{0.4cm}
    \begin{picture}(10.4,3.9)
      \put(0.2,0.2){\circle*{0.45}}
      \put(0.2,2.2){\circle*{0.45}}
      \put(1.2,1.2){\circle*{0.45}}
      \put(0.2,0.2){\line(1,1){1}}
      \put(0.2,2.2){\line(1,-1){1}}
      \put(0.2,0.2){\line(0,1){2}}
      \put(2.2,0.2){\circle*{0.45}}
      \put(2.2,2.2){\circle*{0.45}}
      \put(4.2,0.2){\circle*{0.45}}
      \put(4.2,2.2){\circle*{0.45}}
      \put(2.2,0.2){\line(1,0){2}}
      \put(2.2,0.2){\line(0,1){2}}
      \put(4.2,2.2){\line(-1,0){2}}
      \put(4.2,2.2){\line(0,-1){2}}
      \put(2.2,0.2){\line(1,1){2}}
      \put(5.2,0.2){\circle*{0.45}}
      \put(5.2,2.2){\circle*{0.45}}
      \put(7.2,0.2){\circle*{0.45}}
      \put(7.2,2.2){\circle*{0.45}}
      \put(5.2,0.2){\line(1,0){2}}
      \put(5.2,0.2){\line(0,1){2}}
      \put(7.2,2.2){\line(-1,0){2}}
      \put(7.2,2.2){\line(0,-1){2}}
      \put(8.2,0.2){\circle*{0.45}}
      \put(8.2,2.2){\circle*{0.45}}
      \put(10.2,0.2){\circle*{0.45}}
      \put(10.2,2.2){\circle*{0.45}}
      \put(9.2,1.2){\circle*{0.45}}
      \put(8.2,0.2){\line(1,1){2}}
      \put(8.2,2.2){\line(1,-1){2}}
      \put(8.2,0.2){\line(0,1){2}}
      \put(10.2,0.2){\line(0,1){2}}
    \end{picture}
  \end{center}
\end{Lemma}

\begin{Lemma}
  The complete list of $2$-connected $4$-regular matchstick graphs with $k=5$ is given by\\[-1.2cm]
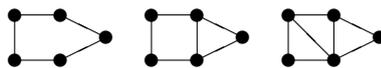
  \begin{center}
    \setlength{\unitlength}{0.6cm}
    \begin{picture}(10.4,2.4)
      \put(0.2,0.2){\circle*{0.3}}
      \put(0.2,1.2){\circle*{0.3}}
      \put(1.2,0.2){\circle*{0.3}}
      \put(1.2,1.2){\circle*{0.3}}
      \put(2.2,0.7){\circle*{0.3}}
      \put(0.2,0.2){\line(1,0){1}}
      \put(0.2,1.2){\line(1,0){1}}
      \put(0.2,0.2){\line(0,1){1}}
      \put(1.2,0.2){\line(2,1){1}}
      \put(1.2,1.2){\line(2,-1){1}}
      \put(3.2,0.2){\circle*{0.3}}
      \put(3.2,1.2){\circle*{0.3}}
      \put(4.2,0.2){\circle*{0.3}}
      \put(4.2,1.2){\circle*{0.3}}
      \put(5.2,0.7){\circle*{0.3}}
      \put(3.2,0.2){\line(1,0){1}}
      \put(3.2,1.2){\line(1,0){1}}
      \put(3.2,0.2){\line(0,1){1}}
      \put(4.2,0.2){\line(2,1){1}}
      \put(4.2,1.2){\line(2,-1){1}}
      \put(4.2,0.2){\line(0,1){1}}
      \put(6.2,0.2){\circle*{0.3}}
      \put(6.2,1.2){\circle*{0.3}}
      \put(7.2,0.2){\circle*{0.3}}
      \put(7.2,1.2){\circle*{0.3}}
      \put(8.2,0.7){\circle*{0.3}}
      \put(6.2,0.2){\line(1,0){1}}
      \put(6.2,1.2){\line(1,0){1}}
      \put(6.2,0.2){\line(0,1){1}}
      \put(7.2,0.2){\line(2,1){1}}
      \put(7.2,1.2){\line(2,-1){1}}
      \put(7.2,0.2){\line(0,1){1}}
      \put(7.2,0.2){\line(-1,1){1}}
    \end{picture}
  \end{center}
\end{Lemma}

\begin{Lemma}
  The complete list of $2$-connected $4$-regular matchstick graphs with $k=6$ is given by
  \begin{center}
    \setlength{\unitlength}{0.6cm}
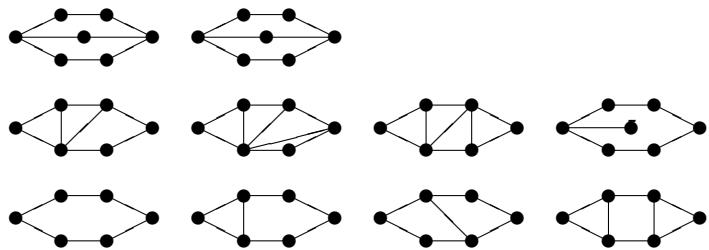
    \begin{picture}(15.4,5.4)
      \put(1.2,1.2){\circle*{0.3}}
      \put(1.2,0.2){\circle*{0.3}}
      \put(2.2,1.2){\circle*{0.3}}
      \put(2.2,0.2){\circle*{0.3}}
      \put(0.2,0.7){\circle*{0.3}}
      \put(3.2,0.7){\circle*{0.3}}
      \put(1.2,0.2){\line(1,0){1}}
      \put(1.2,1.2){\line(1,0){1}}
      \put(2.2,0.2){\line(2,1){1}}
      \put(2.2,1.2){\line(2,-1){1}}
      \put(0.2,0.7){\line(2,1){1}}
      \put(0.2,0.7){\line(2,-1){1}}
      \put(5.2,1.2){\circle*{0.3}}
      \put(5.2,0.2){\circle*{0.3}}
      \put(5.2,0.2){\line(0,1){1}}
      \put(6.2,1.2){\circle*{0.3}}
      \put(6.2,0.2){\circle*{0.3}}
      \put(4.2,0.7){\circle*{0.3}}
      \put(7.2,0.7){\circle*{0.3}}
      \put(5.2,0.2){\line(1,0){1}}
      \put(5.2,1.2){\line(1,0){1}}
      \put(6.2,0.2){\line(2,1){1}}
      \put(6.2,1.2){\line(2,-1){1}}
      \put(4.2,0.7){\line(2,1){1}}
      \put(4.2,0.7){\line(2,-1){1}}
      \put(9.2,1.2){\circle*{0.3}}
      \put(9.2,0.2){\circle*{0.3}}
      \put(10.2,1.2){\circle*{0.3}}
      \put(10.2,0.2){\circle*{0.3}}
      \put(10.2,0.2){\line(-1,1){1}}
      \put(8.2,0.7){\circle*{0.3}}
      \put(11.2,0.7){\circle*{0.3}}
      \put(9.2,0.2){\line(1,0){1}}
      \put(9.2,1.2){\line(1,0){1}}
      \put(10.2,0.2){\line(2,1){1}}
      \put(10.2,1.2){\line(2,-1){1}}
      \put(8.2,0.7){\line(2,1){1}}
      \put(8.2,0.7){\line(2,-1){1}}
      \put(13.2,1.2){\circle*{0.3}}
      \put(13.2,0.2){\circle*{0.3}}
      \put(14.2,0.2){\line(0,1){1}}
      \put(13.2,0.2){\line(0,1){1}}
      \put(14.2,1.2){\circle*{0.3}}
      \put(14.2,0.2){\circle*{0.3}}
      \put(12.2,0.7){\circle*{0.3}}
      \put(15.2,0.7){\circle*{0.3}}
      \put(13.2,0.2){\line(1,0){1}}
      \put(13.2,1.2){\line(1,0){1}}
      \put(14.2,0.2){\line(2,1){1}}
      \put(14.2,1.2){\line(2,-1){1}}
      \put(12.2,0.7){\line(2,1){1}}
      \put(12.2,0.7){\line(2,-1){1}}
      %------------------------------
      \put(1.2,3.2){\circle*{0.3}}
      \put(1.2,2.2){\circle*{0.3}}
      \put(1.2,2.2){\line(0,1){1}}
      \put(1.2,2.2){\line(1,1){1}}
      \put(2.2,3.2){\circle*{0.3}}
      \put(2.2,2.2){\circle*{0.3}}
      \put(0.2,2.7){\circle*{0.3}}
      \put(3.2,2.7){\circle*{0.3}}
      \put(1.2,2.2){\line(1,0){1}}
      \put(1.2,3.2){\line(1,0){1}}
      \put(2.2,2.2){\line(2,1){1}}
      \put(2.2,3.2){\line(2,-1){1}}
      \put(0.2,2.7){\line(2,1){1}}
      \put(0.2,2.7){\line(2,-1){1}}
      \put(5.2,3.2){\circle*{0.3}}
      \put(5.2,2.2){\circle*{0.3}}
      \put(5.2,2.2){\line(0,1){1}}
      \put(5.2,2.2){\line(1,1){1}}
      \put(5.2,2.2){\line(4,1){2}}
      \put(6.2,3.2){\circle*{0.3}}
      \put(6.2,2.2){\circle*{0.3}}
      \put(4.2,2.7){\circle*{0.3}}
      \put(7.2,2.7){\circle*{0.3}}
      \put(5.2,2.2){\line(1,0){1}}
      \put(5.2,3.2){\line(1,0){1}}
      \put(6.2,2.2){\line(2,1){1}}
      \put(6.2,3.2){\line(2,-1){1}}
      \put(4.2,2.7){\line(2,1){1}}
      \put(4.2,2.7){\line(2,-1){1}}
      \put(9.2,3.2){\circle*{0.3}}
      \put(9.2,2.2){\circle*{0.3}}
      \put(9.2,2.2){\line(0,1){1}}
      \put(9.2,2.2){\line(1,1){1}}
      \put(10.2,2.2){\line(0,1){1}}
      \put(10.2,3.2){\circle*{0.3}}
      \put(10.2,2.2){\circle*{0.3}}
      \put(8.2,2.7){\circle*{0.3}}
      \put(11.2,2.7){\circle*{0.3}}
      \put(9.2,2.2){\line(1,0){1}}
      \put(9.2,3.2){\line(1,0){1}}
      \put(10.2,2.2){\line(2,1){1}}
      \put(10.2,3.2){\line(2,-1){1}}
      \put(8.2,2.7){\line(2,1){1}}
      \put(8.2,2.7){\line(2,-1){1}}
      \put(13.2,3.2){\circle*{0.3}}
      \put(13.2,2.2){\circle*{0.3}}
      \put(13.7,2.7){\circle*{0.3}}
      \put(13.7,2.7){\line(-1,1){0.5}}
      \put(13.7,2.7){\line(-1,-1){0.5}}
      \put(13.6,2.7){\line(1-,1){0.5}}
      \put(13.7,2.7){\line(-1,0){1.5}}
      \put(14.2,3.2){\circle*{0.3}}
      \put(14.2,2.2){\circle*{0.3}}
      \put(12.2,2.7){\circle*{0.3}}
      \put(15.2,2.7){\circle*{0.3}}
      \put(13.2,2.2){\line(1,0){1}}
      \put(13.2,3.2){\line(1,0){1}}
      \put(14.2,2.2){\line(2,1){1}}
      \put(14.2,3.2){\line(2,-1){1}}
      \put(12.2,2.7){\line(2,1){1}}
      \put(12.2,2.7){\line(2,-1){1}}
      %------------------------------
      \put(1.2,5.2){\circle*{0.3}}
      \put(1.2,4.2){\circle*{0.3}}
      \put(1.7,4.7){\circle*{0.3}}
      \put(1.7,4.7){\line(-1,1){0.5}}
      \put(1.7,4.7){\line(-1,-1){0.5}}
      \put(0.2,4.7){\line(1,0){3}}
      \put(2.2,5.2){\circle*{0.3}}
      \put(2.2,4.2){\circle*{0.3}}
      \put(0.2,4.7){\circle*{0.3}}
      \put(3.2,4.7){\circle*{0.3}}
      \put(1.2,4.2){\line(1,0){1}}
      \put(1.2,5.2){\line(1,0){1}}
      \put(2.2,4.2){\line(2,1){1}}
      \put(2.2,5.2){\line(2,-1){1}}
      \put(0.2,4.7){\line(2,1){1}}
      \put(0.2,4.7){\line(2,-1){1}}
      \put(5.2,5.2){\circle*{0.3}}
      \put(5.2,4.2){\circle*{0.3}}
      \put(5.7,4.7){\circle*{0.3}}
      \put(5.7,4.7){\line(-1,-1){0.5}}
      \put(5.7,4.7){\line(1,1){0.5}}
      \put(4.2,4.7){\line(1,0){3}}
      \put(6.2,5.2){\circle*{0.3}}
      \put(6.2,4.2){\circle*{0.3}}
      \put(4.2,4.7){\circle*{0.3}}
      \put(7.2,4.7){\circle*{0.3}}
      \put(5.2,4.2){\line(1,0){1}}
      \put(5.2,5.2){\line(1,0){1}}
      \put(6.2,4.2){\line(2,1){1}}
      \put(6.2,5.2){\line(2,-1){1}}
      \put(4.2,4.7){\line(2,1){1}}
      \put(4.2,4.7){\line(2,-1){1}}
    \end{picture}
  \end{center}
\end{Lemma}

Extending these results requires quite some effort. During his diploma thesis \cite{dipl_achim} Achim Hildenbrandt was able to prove:

\begin{Lemma}
  Up to symmetry there are exactly $34$ two-connected $4$-regular matchstick graphs with $k=7$. All of them contain at most two
  inner vertices.
\end{Lemma}

Here a combinatorial question arises: What is the maximum number $f_4(k)$ of inner vertices of a $4$-regular matchstick graphs with given cardinality $k$ of its outer face? In this context we would like to remark that the corresponding number $f_3(k)$ for $3$-regular matchstick graphs is unbounded for $k\ge 6$ \cite{girth_four}. 
%In the remaining part of this section we will prove that $f_4(k)$ can be bounded by a polynomial in $k$.

\begin{Definition}
  Let $\mathcal{M}$ be a $4$-regular matchstick graph. Let $Q$ be the graph arising from $\mathcal{M}$, where
  the vertices correspond to the quadrangles from $\mathcal{M}$. Two vertices are connected by an edge iff the 
  corresponding quadrangles have a face in common. A \textbf{quadrangle component} of $\mathcal{M}$ is given by
  the inverse map of the connected components of $Q$.
\end{Definition}

\begin{Lemma}
  A quadrangle component $\mathcal{Q}$ consisting of $q$ quadrangles has an outer face consisting of at least
  $2\left\lceil 2\sqrt{q}\,\right\rceil$ edges and vertices.
\end{Lemma}
\begin{proof}
  The proof in \cite{0402.05055} also works for quadrangle components.
\end{proof}

\begin{Lemma}
  \label{lemma_1}
  Let $\mathcal{M}$ be a $4$-regular matchstick graph with $k\ge 5$, where the quadrangle components consist of
  $q_1,\dots,q_s$ quadrangles, then we have
  $3A_3+\sum\limits_{i=5}^k iA_i \ge\sum\limits_{i=1}^s 2\left\lceil 2\sqrt{q_i}\,\right\rceil$ and $\sum\limits_{i=1}^s q_i=A_4$.
\end{Lemma}
\begin{proof}
  Every vertex $v$ on the outer face of a quadrangle component can be part of at most two quadrangle components. In case of
  equality vertex $v$ is adjacent to two non-quadrangular faces. In the other case $v$ is adjacent to at least one
  non-quadrangular face.
\end{proof}

\begin{Lemma}
  \label{lemma_2}
  For a $4$-regular matchstick graph we have $A_4\le \frac{k^4}{4}$.
\end{Lemma}
\begin{proof}
  A $k$-gon with side lengths $1$ has an area of at most $\frac{k^2}{4\pi}$. Combining this with Lemma~\ref{lemma_area_argument}
  we obtain $\sum\limits_{i=1}^\infty A_{2i+1} \le \frac{k^2}{\sqrt{3}\pi}$. Subtracting (\ref{eq_A_i_sum}) yields
  \begin{equation}
    \label{eq_t_2}
    \sum_{i=5}^\infty 2\cdot\left\lfloor\frac{i-3}{2}\right\rfloor\cdot A_i \le
    \frac{k^2}{\sqrt{3}\pi}-8+\tau\le \frac{k^2}{\sqrt{3}\pi}-8+2k\le\frac{2k^2}{5}.
  \end{equation}
  Adding the last two inequalities and multilying by three results in
  \begin{equation}
    \label{eq_t_3}
    3A_3+\sum_{i=5}^k iA_i\le3\sum_{i=1}^\infty A_{2i+1}+\sum_{i=5}^\infty 6\cdot\left\lfloor\frac{i-3}{2}\right\rfloor\cdot A_i
    \le 2k^2.
  \end{equation}
  From Lemma~\ref{lemma_1} we deduce
  $
    2k^2 \ge3A_3+\sum\limits_{i=5}^k iA_i \ge 4\sqrt{A_4}
  $
  and conclude $A_4\le \frac{k^4}{4}$.
\end{proof}

\begin{Theorem}
  $$
    f_4(k)\le \frac{2k^2}{3}+\frac{k^4}{4}-2.
  $$
\end{Theorem}
\begin{proof}
  From Inequality~(\ref{eq_t_3}) and Lemma~\ref{lemma_2} we conclude
  $$
    |F|=\sum_{i=3}^\infty A_i\le \frac{2k^2}{3}+\frac{k^4}{4}.
  $$
  Applying $n=|V|=|F|-2+\frac{\tau}{2}$ from Lemma~\ref{lemma_A_i_sum} and using $\tau\le 2k$ yields the proposed inequality.
\end{proof}

\end{document}